\documentclass[reqno,12pt]{amsart}

\usepackage[utf8]{inputenc}
\usepackage{amsfonts}
\usepackage{amsmath}
\usepackage{graphicx}
\usepackage{float}
\usepackage{color}
\usepackage{hyperref}
\usepackage{tikz-cd}

\usepackage{amssymb}
\usepackage{enumitem}
\usepackage{mathrsfs}

\usepackage{tikz}
\usetikzlibrary{topaths}
\usetikzlibrary{calc}

\usepackage{a4wide}

\usepackage{empheq}

\newtheorem{theorem}{Theorem}[section]
\newtheorem{lemma}[theorem]{Lemma}
\newtheorem{proposition}[theorem]{Proposition}

\newtheorem{cit}[theorem]{Citation}
\newtheorem{observation}[theorem]{Observation}

\newtheorem*{main:inherit_ICC}{Theorem~\ref{thrm:inherit_ICC}}
\newtheorem*{main:diverse_ICC}{Theorem~\ref{thrm:diverse_ICC}}
\newtheorem*{main:thomp_mcduff}{Theorem~\ref{thrm:thomp_mcduff}}

\theoremstyle{definition}
\newtheorem{definition}[theorem]{Definition}
\newtheorem{remark}[theorem]{Remark}
\newtheorem{example}[theorem]{Example}
\newtheorem{question}[theorem]{Question}

\newcommand{\Z}{\mathbb{Z}}
\newcommand{\N}{\mathbb{N}}

\newcommand{\C}{\mathbb{C}}

\newcommand{\clone}{\kappa}
\newcommand{\symmclone}{\varsigma}
\newcommand{\vN}{\mathcal{L}}
\newcommand{\tree}{\mathcal{T}}

\newcommand{\defeq}{\mathbin{\vcentcolon =}}

\DeclareMathOperator{\Aut}{Aut}

\DeclareMathOperator{\GL}{GL}
\DeclareMathOperator{\Ab}{Ab}
\DeclareMathOperator{\F}{F}

\DeclareMathOperator{\II}{II}

\DeclareMathOperator{\trace}{tr}

\newcommand{\Thomp}{\mathscr{T}}
\newcommand{\Thkern}{\mathscr{K}}

\numberwithin{equation}{section}

\begin{document}

\title{Von Neumann algebras of Thompson-like groups from cloning systems}
\date{\today}


\author[E.~Bashwinger]{Eli Bashwinger}
\address{Department of Mathematics and Statistics, University at Albany (SUNY), Albany, NY 12222}
\email{ebashwinger@albany.edu}

\author[M.~C.~B.~Zaremsky]{Matthew C.~B.~Zaremsky}
\address{Department of Mathematics and Statistics, University at Albany (SUNY), Albany, NY 12222}
\email{mzaremsky@albany.edu}

\begin{abstract}
We prove a variety of results about the group von Neumann algebras associated to Thompson-like groups arising from so called $d$-ary cloning systems. Cloning systems are a framework developed by Witzel and the second author, with a $d$-ary version subsequently developed by Skipper and the second author, which can be used to construct generalizations of the classical Thompson's groups $F$, $T$, and $V$. Given a family of groups $(G_n)_{n\in\N}$ with a $d$-ary cloning system, we get a Thompson-like group $\Thomp_d(G_*)$, and in this paper we find some mild, natural conditions under which the group von Neumann algebra $\vN(\Thomp_d(G_*))$ has desirable properties. For instance, if the $d$-ary cloning system is ``fully compatible'' and ``diverse'' then we prove that $\vN(\Thomp_d(G_*))$ is a type $\II_1$ factor. If moreover the $d$-ary cloning system is ``uniform'' and ``slightly pure'' then we prove $\vN(\Thomp_d(G_*))$ is even a McDuff factor, so $\Thomp_d(G_*)$ is inner amenable. Examples of $d$-ary cloning systems satisfying these conditions are easy to come by, and include many existing examples, for instance our results show that for $bV$ and $bF$ the Brin--Dehornoy braided Thompson group and pure braided Thompson group, $\vN(bV)$ and $\vN(bF)$ are type $\II_1$ factors and $\vN(bF)$ is McDuff. In particular we get the surprising result that $bF$ is inner amenable.
\end{abstract}

\maketitle
\thispagestyle{empty}

\section*{Introduction}

The main goal of this paper is to produce an array of new examples of group von Neumann algebras that are type $\II_1$ factors, and even McDuff factors. The examples come from groups denoted $\Thomp_d(G_*)$, which are members of the extended family of Thompson's groups, and arise by finding a so called $d$-ary cloning system on a family of groups $(G_n)_{n\in\N}$. Cloning systems were developed by the second author with Stefan Witzel in \cite{witzel18} (see also \cite{witzel19} and \cite{zaremsky18user}), and generalized to $d$-ary cloning systems by the second author and Rachel Skipper in \cite{skipper21}. They have since proved useful in a variety of contexts, for example producing simple groups separated by finiteness properties \cite{skipper19}, inspecting inheritance properties of (bi-)orderability \cite{ishida18}, and producing potential counterexamples to the conjecture that every co$\mathcal{CF}$ group embeds into Thompson's group $V$ \cite{berns-zieve18}.

The group von Neumann algebras of the classical Thompson groups $F$, $T$, and $V$ are all type $\II_1$ factors, as the groups are ICC (meaning every non-trivial conjugacy class is infinite). The group von Neumann algebra of $F$ was shown to even be a McDuff factor by Jolissaint in \cite{jolissaint98}, and so $F$ is inner amenable. This general picture informs our results. Thompson-like groups arising from $d$-ary cloning systems tend to be ``$F$-like'' or ``$V$-like'' (or somewhere in between, although to some extent this should also be considered $V$-like). We show that any $\Thomp_d(G_*)$ satisfying some mild hypotheses is ICC and so yields a type $\II_1$ factor, and the $F$-like ones even yield McDuff factors. In particular this shows that the $F$-like ones are inner amenable. Note that $T$ and $V$ are not inner amenable, as proved by Haagerup--Olesen \cite{haagerup17}.

More precisely, our main results are as follows. First we see that, in the ``fully compatible'' case (Definition~\ref{def:fully_compatible}), if the $G_n$ themselves are ICC then $\Thomp_d(G_*)$ is too:

\begin{main:inherit_ICC}
Let $((G_n)_{n\in\N},(\rho_n)_{n\in\N},(\clone_k^n)_{k\le n})$ be a fully compatible $d$-ary cloning system. If all the $G_n$ are ICC then so is $\Thomp_d(G_*)$.
\end{main:inherit_ICC}

Even if the $G_n$ are not ICC (as in many easy examples), adding a natural hypothesis that the $d$-ary cloning system is ``diverse'' (Definition~\ref{def:diverse}) implies that $\Thomp_d(G_*)$ is nonetheless ICC:

\begin{main:diverse_ICC}
Let $((G_n)_{n\in\N},(\rho_n)_{n\in\N},(\clone_k^n)_{k\le n})$ be a $d$-ary cloning system. Assume that it is fully compatible and diverse. Then $\Thomp_d(G_*)$ is ICC.
\end{main:diverse_ICC}

Finally, in the fully compatible, ``slightly pure'' case (Definition~\ref{def:slightly_pure}), an additional natural assumption called ``uniformity'' (Definition~\ref{def:uniform}) ensures that the group von Neumann algebra $\vN(\Thomp_d(G_*))$ of $\Thomp_d(G_*)$ is McDuff:

\begin{main:thomp_mcduff}
Let $((G_n)_{n\in\N},(\rho_n)_{n\in\N},(\clone_k^n)_{k\le n})$ be a fully compatible, slightly pure, uniform $d$-ary cloning system. Assume either that all the $G_n$ are ICC, or that the $d$-ary cloning system is diverse (so in either case $\Thomp_d(G_*)$ is ICC). Then $\vN(\Thomp_d(G_*))$ is a McDuff factor and $\Thomp_d(G_*)$ is inner amenable.
\end{main:thomp_mcduff}

It is notable that whether the $G_n$ are ICC or not, or whether the $\vN(G_n)$ are McDuff or not, does not matter much, as long as these mild conditions on the $d$-ary cloning systems are met. For example, we prove that the braided Thompson group $bF$ (see Subsection~\ref{sec:braid}) is inner amenable and $\vN(bF)$ is a McDuff factor, even though pure braid groups are not even ICC. We also prove that a close relative of $V$ denoted $\widehat{V}$ (Example~\ref{ex:Vhat}) is inner amenable, somewhat surprisingly.

We remark that groups yielding McDuff factors are somewhat rare among groups studied in geometric group theory, so Thompson-like groups are unusual in this regard. For example, a large class of groups studied in geometric group theory are acylindrically hyperbolic, and Dahmani--Guirardel--Osin proved in \cite[Theorem~8.14]{dahmani17} that acylindircally hyperbolic ICC groups cannot be inner amenable, and so cannot yield McDuff type $\II_1$ factors.

This paper is organized as follows. In Section~\ref{sec:vN} we recall the relevant background on group von Neumann algebras, type $\II_1$ factors, and McDuff factors, and in Section~\ref{sec:cloning} we recall the background on $d$-ary cloning systems. In Section~\ref{sec:cloning_ICC} we prove Theorems~\ref{thrm:inherit_ICC} and~\ref{thrm:diverse_ICC}, and in Section~\ref{sec:examples} we present many resulting examples of type $\II_1$ factors. In Section~\ref{sec:mcduff} we prove Theorem~\ref{thrm:thomp_mcduff}, and present many resulting examples of McDuff factors. Finally, in Section~\ref{sec:questions} we pose a number of naturally arising questions.

\subsection*{Acknowledgments} We are very grateful to Jon Bannon for myriad helpful discussions on the general theory of von Neumann algebras. The second author is supported by grant \#635763 from the Simons Foundation.

\section{Group von Neumann algebras}\label{sec:vN}

In this section we recall some background on group von Neumann algebras. Unless specified otherwise, everything in this section is drawn from \cite{jolissaint98} and \cite{picioroaga06}. A \emph{von Neumann algebra} $M$ is a $*$-subalgebra of bounded linear operators on some Hilbert space that is closed with respect to the weak operator topology. A robust source of von Neumann algebras is given by the group von Neumann algebra $\vN(G)$ of a countable discrete group $G$. (In what follows, not every result will be dependent on $G$ being countable, but enough of them will be that we will just impose a general rule that all groups we consider here are countable.) This is constructed as follows. Let $\ell^2(G)$ be the Hilbert space
\[
\ell^2(G)\defeq \{\psi\colon G\to\C\mid \sum\limits_{g\in G}|\psi(g)|^2<\infty\}
\]
with inner product $\langle \phi,\psi\rangle \defeq \sum\limits_{g\in G}\phi(g)\overline{\psi(g)}$. Clearly $G$ acts faithfully on $\ell^2(G)$ by bounded linear (unitary) operators, so $\C[G]$ is a subalgebra of the algebra $\mathcal{B}(\ell^2(G))$ of bounded linear operators.

\begin{definition}[Group von Neumann algebra]
The \emph{group von Neumann algebra} $\vN(G)$ of $G$ is the closure of $\C[G]$ in $\mathcal{B}(\ell^2(G))$ with respect to the weak operator topology.
\end{definition}

Let $\delta_e \in \ell^2(G)$ be the usual function sending the identity $e\in G$ to $1$ and all other $g$ to $0$. Define a faithful, normal trace $\trace\colon \vN(G)\to \C$ via $\trace(L)\defeq \langle L(\delta_e),\delta_e\rangle$. From this we get the Hilbertian norm $||L||_2\defeq \sqrt{\trace(L^* L)}$. Note that for any $e\ne g\in G$ we have $\trace(g)=0$, and for any $g\ne h$ in $G$ we have $||g-h||_2=\sqrt{2}$.

\subsection{Type $\II_1$ factors and ICC}\label{sec:ICC}

An especially important kind of von Neumann algebra is the following:

\begin{definition}[Type $\II_1$ factor]
A von Neumann algebra $M$ is called a \emph{factor} if its center is trivial, meaning the only elements of $M$ commuting with every element of $M$ are the scalar multiples of the identity operator. A factor is of \emph{type $\II_1$} if it is infinite dimensional and admits a trace state. See, e.g., \cite{popa07} for many more details on type $\II_1$ factors, and their applications in a wide array of fields.
\end{definition}

It turns out (given that we are restricting to countable groups) that $\vN(G)$ is a type $\II_1$ factor if and only if $G$ is ICC, defined as follows.

\begin{definition}
A group is said to be \emph{ICC} if the conjugacy class of every non-trivial element is infinite. Since an element has finite conjugacy class if and only if its centralizer has finite index, a group being ICC is equivalent to the property that the only element whose centralizer has finite index is the identity. This is turn is equivalent to the property that every finite index subgroup has trivial center.
\end{definition}

Let us discuss how the ICC property behaves under group extensions. The following criterion will be very useful for our situation, and is obtained by combining two results from \cite{preaux13}.

\begin{lemma}\label{lem:preaux_criterion}
Let $G$ be a group and $K$ a normal subgroup of $G$. Suppose that $G/K$ is ICC and that the following holds: for any subgroup $N$ of $K$ that is normal in $G$, if either $N$ is finite, or if $N$ is isomorphic to $\Z^n$ ($n\ge 0$) and the map $G\to\GL_n(\Z)$ induced by the conjugation action of $G$ on $N$ has finite image, then $N$ is trivial. (For example if the only finitely generated subgroup of $K$ normal in $G$ is the trivial one, then this condition holds.) Then $G$ is ICC.
\end{lemma}

\begin{proof}
Following \cite{preaux13}, let $FC_G(K)\defeq \{g\in K\mid g$ has finitely many conjugates by elements of $G\}$. Our assumptions ensure that $FC_G(K)=\{1\}$, by \cite[Proposition~1.1]{preaux13}. Now consider \cite[Theorem~2.3]{preaux13}, with $Q=G/K$. Since $FC_G(K)=\{1\}$, condition (i) is satisfied. Since $Q$ is ICC, condition (ii) is also satisfied for trivial reasons.
\end{proof}

\subsection{McDuff factors}\label{sec:mcduff_def}

Let us conclude this section with a discussion of McDuff factors, central sequences, and inner amenability.

\begin{definition}[(Relative) McDuff]
A type $\II_1$ factor $M$ is \emph{McDuff} if $M\cong M\otimes R$, where $R$ is the hyperfinite $\II_1$ factor, e.g., the von Neumann algebra of $S_\infty$ (or of any non-trivial amenable ICC group). A pair of $\II_1$ factors $N\subseteq M$ has the \emph{relative McDuff property} if there is an isomorphism $M\to M\otimes R$ restricting to an isomorphism $N\to N\otimes R$.
\end{definition}

Note that if a pair $N\subseteq M$ has the relative McDuff property then in particular each of $M$ and $N$ is McDuff. The McDuff property can also be phrased in terms of central sequences.

\begin{definition}[Central sequence]
Let $M$ be a type $\II_1$ factor. A sequence $(a_n)_{n\in\N} \in \ell^\infty(\N,M)$ is called \emph{central} if
\[
\lim_{n\to\infty}||aa_n-a_n a||_2=0
\]
for all $a\in M$. Two central sequences $(a_n)_{n\in\N}$ and $(b_n)_{n\in\N}$ are \emph{equivalent} if
\[
\lim_{n\to\infty}||a_n-b_n||_2=0 \text{.}
\]
A central sequence is \emph{trivial} if it is equivalent to a scalar sequence.
\end{definition}

It turns out $M$ is McDuff if and only if it admits a pair of non-commuting, non-trivial central sequences \cite{mcduff70}. Related to the McDuff property is property $\Gamma$ of Murray and von Neumann \cite{murray43}:

\begin{definition}[Property $\Gamma$]
We say $M$ has \emph{property $\Gamma$} if there exists a non-trivial central sequence.
\end{definition}

In particular McDuff implies property $\Gamma$. For an ICC group $G$, if $\vN(G)$ has property $\Gamma$ then $G$ is inner amenable \cite{effros75}, meaning there is a conjugation-invariant mean on $G\setminus\{e\}$. For completeness, we remark that one can define inner amenability without restricting to ICC groups as follows:

\begin{definition}[Inner amenable]
A group is called \emph{inner amenable} if it admits an atomless, conjugation-invariant mean.
\end{definition}

\section{Cloning systems}\label{sec:cloning}

In \cite{witzel18}, the second author and Witzel introduced the notion of a \emph{cloning system} on a family of groups $(G_n)_{n\in\N}$. Given a family $(G_n)_{n\in\N}$ with a cloning system, one can construct a Thompson-like group $\Thomp(G_*)$ that contains all the $G_n$ as natural subgroups in a particularly nice way. In particular, in many examples, finiteness properties of the $G_n$ carry over to $\Thomp(G_*)$.

In \cite{skipper21}, the second author and Skipper expanded the construction to \emph{$d$-ary cloning systems}, which is the generality we will use here (when $d=2$ the original construction is recovered). Let us recall the definition here.

\begin{definition}[$d$-ary cloning system]\label{def:cloning}
Let $d\ge 2$ be an integer and $(G_n)_{n\in\N}$ be a family of groups. For each $n\in\N$ let $\rho_n\colon G_n\to S_n$ be a homomorphism to the symmetric group $S_n$, called a \emph{representation map}. For each $1\le k\le n$ let $\clone_k^n \colon G_n \to G_{n+d-1}$ be an injective function (not necessarily a homomorphism), called a \emph{$d$-ary cloning map}. We write $\rho_n$ to the left of its input and $\clone_k^n$ to the right of its input, for reasons of visual clarity. Now we call the triple
\[
((G_n)_{n\in\N},(\rho_n)_{n\in\N},(\clone_k^n)_{k\le n})
\]
a \emph{$d$-ary cloning system} if the following axioms hold:\\
\indent\textbf{(C1):} (Cloning a product) $(gh)\clone_k^n = (g)\clone_{\rho_n(h)k}^{n} (h)\clone_k^n$\\
\indent\textbf{(C2):} (Product of clonings) $\clone_\ell^n \circ \clone_k^{n+d-1} = \clone_k^n \circ \clone_{\ell+d-1}^{n+d-1}$\\
\indent\textbf{(C3):} (Compatibility) $\rho_{n+d-1}((g)\clone_k^n)(i) = (\rho_n(g))\symmclone_k^n(i)$ for all $i\ne k,k+1,\dots,k+d-1$.

Here we always have $1\le k<\ell\le n$ and $g,h\in G_n$, and $\symmclone_k^n$ denotes the standard $d$-ary cloning maps for the symmetric groups, explained in \cite[Example~2.2]{skipper21}.
\end{definition}

Given a $d$-ary cloning system on a family of groups $(G_n)_{n\in\N}$ one gets a Thompson-like group, denoted $\Thomp_d(G_*)$, which can be viewed as a sort of ``Thompson-esque limit'' of the $G_n$. (When $d=2$ we will still write $\Thomp(G_*)$ for $\Thomp_2(G_*)$.) Let us recall the construction of $\Thomp_d(G_*)$. First, a \emph{$d$-ary tree} is a finite rooted tree in which each non-leaf vertex has $d$ children, and a \emph{$d$-ary caret} is a $d$-ary tree with $d$ leaves. An element of $\Thomp_d(G_*)$ is represented by a triple $(T_-,g,T_+)$ where $T_\pm$ are $d$-ary trees with the same number of leaves, say $n$, and $g$ is an element of $G_n$. In the future we may say ``tree'' instead of ``$d$-ary tree'' when the context is clear.

There is an equivalence relation on such triples, whose equivalence classes are the elements of $\Thomp_d(G_*)$. The equivalence relation is given by expansion and reduction: an \emph{expansion} of $(T_-,g,T_+)$ is a triple of the form $(T_-',(g)\clone_k^n,T_+')$ where $T_+'$ is $T_+$ with a $d$-ary caret added to the $k$th leaf and $T_-'$ is $T_-$ with a $d$-ary caret added to the $\rho_n(g)(k)$th leaf. A \emph{reduction} is the reverse of an expansion. Let us also recursively extend the definition of \emph{expansion (reduction)} to mean the result of any finite sequence of expansions (reductions). Given a tree $T$, let us also call an \emph{expansion} of $T$ the result of iteratively adding $d$-carets to leaves, finitely many times. Now declare that two triples are equivalent if we can get from one to the other via a finite sequence of expansions and reductions, and write $[T_-,g,T_+]$ for the equivalence class of $(T_-,g,T_+)$.

The group $\Thomp_d(G_*)$ is the set of equivalence classes $[T_-,g,T_+]$, but we have not yet explained the group operation. The idea is that given any two elements $[T_-,g,T_+]$ and $[U_-,h,U_+]$, up to expansions we can assume $T_+=U_-$. This is because any pair of $d$-ary trees have a common $d$-ary tree obtainable from either of them by adding $d$-ary carets to their leaves. Now the group operation on $\Thomp_d(G_*)$ is defined by
\[
[T_-,g,T_+][U_-,h,U_+] \defeq [T_-,gh,U_+]
\]
when $T_+=U_-$. The cloning axioms ensure that this is a well defined group operation. The identity is $[T,1,T]$ (for any $T$) and inverses are given by $[T_-,g,T_+]^{-1} = [T_+,g^{-1},T_-]$.

\begin{example}
The easiest examples of $d$-ary cloning systems are those yielding the Thompson-like groups $F_d$ and $V_d$. Here $V_d$ is the well known ``$d$-ary'' Higman--Thompson group, and $F_d$ is its ``$F$-like'' subgroup. If every $G_n$ is the trivial group $\{1\}$ (so the $\rho_n$ and $\clone_k^n$ are trivial too) then the $d$-ary Thompson-like group $\Thomp_d(\{1\})$ is $F_d$. Elements of $F_d$ are equivalence classes of the form $[T_-,1,T_+]$, or we can just write $[T_-,T_+]$. If $G_n=S_n$, $\rho_n$ is the identity, and $\clone_k^n=\symmclone_k^n$ for all $1\le k\le n$, then the $d$-ary cloning system $\Thomp_d(S_*)$ is $V_d$. Elements of $V_d$ are equivalence classes of the form $[T_-,\sigma,T_+]$ for $\sigma\in S_n$. (Doing the same thing restricted to $\Z/n\Z\cong \langle(1~2~\cdots~n)\rangle\le S_n$ yields the Higman--Thompson group $T_d$.) See Figure~\ref{fig:V} for an example of expansion using the ($2$-ary) cloning system yielding $V=V_2$.
\end{example}

\begin{figure}[htb]
 \begin{tikzpicture}[line width=1pt]\centering
  \draw (0,0) -- (2,2) -- (4,0)   (3,1) -- (2,0);
	
	\draw[dashed] (2,-3) -- (0,0);
	
	\draw[dashed] (0,-3) -- (4,0);
	
  \draw[dashed] (4,-3) -- (2,0);
	
	\filldraw (0,0) circle (1.5pt);
	\filldraw (2,0) circle (1.5pt);
	\filldraw (4,0) circle (1.5pt);
	
	\draw (0,-3) -- (2,-5) -- (4,-3)   (1,-4) -- (2,-3);
	\filldraw (0,-3) circle (1.5pt);
	\filldraw (2,-3) circle (1.5pt);
	\filldraw (4,-3) circle (1.5pt);
	
	\node at (5.5,-1.5) {$\longrightarrow$};
	
	\begin{scope}[xshift=8cm,yshift=1cm]	
   \draw (-1,-1) -- (2,2) -- (5,-1)   (3,1) -- (2,0)   (1,-1) -- (2,0) -- (3,-1);
	 
   \draw[dashed] (1,-4) -- (-1,-1);
	
   \draw[dashed] (-1,-4) -- (5,-1);
	
   \draw[dashed] (3,-4) -- (1,-1);
	
   \draw[dashed] (5,-4) -- (3,-1);
	
	 \filldraw (-1,-1) circle (1.5pt);
	 \filldraw (1,-1) circle (1.5pt);
	 \filldraw (3,-1) circle (1.5pt);
	 \filldraw (5,-1) circle (1.5pt);
	
	 \draw (-1,-4) -- (2,-7) -- (5,-4)   (0,-5) -- (1,-4)   (3,-4) -- (4,-5);
	 \filldraw (-1,-4) circle (1.5pt);
	 \filldraw (1,-4) circle (1.5pt);
	 \filldraw (3,-4) circle (1.5pt);
	 \filldraw (5,-4) circle (1.5pt);
	\end{scope}
 \end{tikzpicture}
 \caption{An example of an expansion using the ($2$-ary) cloning system defined above on $(S_n)_{n\in\N}$. Here we draw $(T_-,\sigma,T_+)$ by drawing $T_+$ upside-down and below $T_-$, with the permutation $\sigma$ indicated by dashed lines connecting the leaves. This makes the expansion move look quite natural. Since these pictures differ by an expansion, they represent the same element of $V=\Thomp_2(S_*)$.}\label{fig:V}
\end{figure}
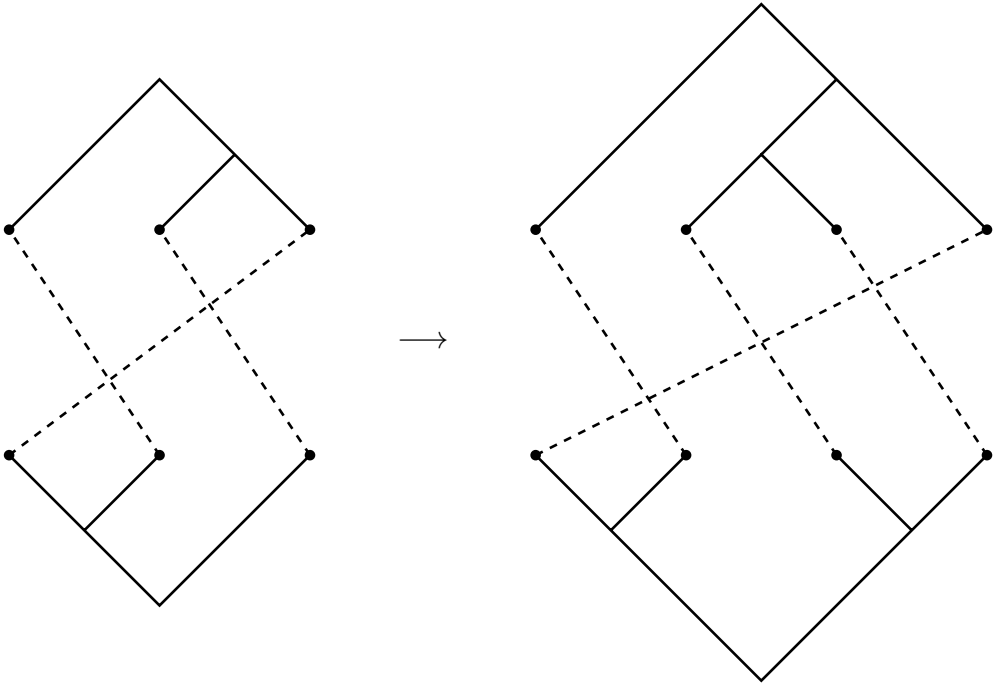

\begin{lemma}\label{lem:W_ICC}
Let $W_d$ be a subgroup of the Higman--Thompson group $V_d$ containing the commutator subgroup $[F_d,F_d]$. Then $W_d$ is ICC.
\end{lemma}

\begin{proof}
We need to show that if the centralizer in $W_d$ of $g\in W_d$ has finite index in $W_d$, then $g=1$. Since $[F_d,F_d]$ is simple (see, e.g., \cite[Theorem~4.13]{brown87}), it has no proper finite index subgroups, so any such $g$ commutes with every element of $[F_d,F_d]$. Consider the action of $V_d$ on $[0,1)$ by right-continuous bijections (see \cite[Section~6]{cannon96}). For any $a\in (0,1)\cap\Z[\frac{1}{d}]$ we can choose $b\in (0,1)\cap\Z[\frac{1}{d}]$ and $f\in [F_d,F_d]$ such that the support of $f$ in $(0,1)$ (that is, the set of non-fixed points) is precisely $(a,b)$. Since $g$ commutes with $f$, $g$ must stabilize $(a,b)$, and since $g$ acts by right-continuous bijections it must therefore fix $a$. Since $a$ was arbitrary, $g$ fixes every element of $(0,1)\cap\Z[\frac{1}{d}]$. The only element of $V_d$ doing this is the identity, so we are done.
\end{proof}

\subsection{Fully compatible and pure}\label{sec:pure}

Note that in the ``compatibility'' axiom in the definition of $d$-ary cloning system, we do not require the condition to hold for $i=k,k+1,\dots,k+d-1$. In some natural examples in fact it does not hold, e.g., for the $d$-ary cloning systems in \cite{skipper21} producing R\"over--Nekrashevych groups. However, assuming that the compatibility condition holds for all $i$ leads to some nice properties, e.g., the existence of a natural map $\Thomp_d(G_*)\to V_d$. Let us encode this into the following definition.

\begin{definition}[Fully compatible]\label{def:fully_compatible}
Call a $d$-ary cloning system \emph{fully compatible} if we have $\rho_{n+d-1}((g)\clone_k^n)(i) = (\rho_n(g))\symmclone_k^n(i)$ for all $1\le k\le n$ and all $1\le i\le n+d-1$ (even if $i=k,k+1,\dots,k+d-1$).
\end{definition}

Given a fully compatible $d$-ary cloning system, if $g\in\ker(\rho_n)$ then $(g)\clone_k^n \in \ker(\rho_{n+d-1})$ for any $1\le k\le n$. In particular, given a triple of the form $(T,g,T)$ for $g\in \ker(\rho_{n(T)})$ (here $n(T)$ is the number of leaves of $T$), if $T'$ is the tree obtained from $T$ by adding a $d$-caret to the $k$th leaf, then the expansion $(T',(g)\clone_k^n,T')$ of $(T,g,T)$ is again of this form. In particular this shows that the following is a subgroup of $\Thomp_d(G_*)$:
\[
\Thkern_d(G_*)\defeq \{[T,g,T]\mid g\in\ker(\rho_{n(T)})\} \text{.}
\]

\begin{lemma}[Map to $V_d$]\label{lem:map_to_Vd}
Given a fully compatible $d$-ary cloning system as above, there is a map $\pi\colon \Thomp_d(G_*)\to V_d$ given by sending $[T_-,g,T_+]$ to $[T_-,\rho_n(g),T_+]$, where $g\in G_n$. The kernel is $\Thkern_d(G_*)$ and the image is some group $W_d$ with $F_d\le W_d\le V_d$.
\end{lemma}

\begin{proof}
When $d=2$ this is \cite[Lemma~3.2]{witzel18}, and the general case works analogously.
\end{proof}

Note that the ``cloning a product'' axiom ensures that $(1)\clone_k^n=1$ for all $k$ and $n$, so any triple equivalent to $(T_-,1,T_+)$ is of the form $(T_-',1,T_+')$, i.e., the middle entry being the identity is an invariant of the equivalence relation. This implies that the set of elements of $\Thomp_d(G_*)$ of the form $[T_-,1,T_+]$ forms a subgroup, and it is easy to see that it is isomorphic to the generalized Thompson group $F_d$. Note that in the fully compatible situation, the image of this copy of $F_d$ in $\Thomp_d(G_*)$ maps under $\pi$ isomorphically to the standard copy of $F_d$ inside of $W_d$.

The easiest way to ensure full compatibility is if all the $\rho_n$ are trivial, which gives us the following definition:

\begin{definition}[Pure]\label{def:pure}
Call a $d$-ary cloning system \emph{pure} if every $\rho_n\colon G_n\to S_n$ is trivial.
\end{definition}

It is clear that a pure $d$-ary cloning system is fully compatible, and that in this case $\pi\colon \Thomp_d(G_*)\to V_d$ has image $F_d$ and splits, i.e., $\Thomp_d(G_*)=\Thkern_d(G_*)\rtimes F_d$.

We will sometimes informally refer to Thompson-like groups arising from pure $d$-ary cloning systems as ``$F$-like'' and those arising from $d$-ary cloning systems in which the $\rho_n\colon G_n\to S_n$ are surjective as ``$V$-like''.

\section{Cloning systems and ICC}\label{sec:cloning_ICC}

In this section we present some sufficient conditions under which a Thompson-like group $\Thomp_d(G_*)$ is ICC. The main results are Theorems~\ref{thrm:inherit_ICC} and~\ref{thrm:diverse_ICC}. We also discuss some natural situations in which $\Thomp_d(G_*)$ is not ICC.

\subsection{ICC Thompson-like groups}\label{sec:icc_examples}

First let us show that if the $G_n$ happen to be ICC and full compatibility holds, then $\Thomp_d(G_*)$ is ICC.

For a tree $T$, let $n(T)$ be the number of leaves of $T$. Define the subgroup
\[
G_T\defeq \{[T,g,T]\mid g\in G_{n(T)}\} \le \Thomp_d(G_*)\text{.}
\]
Note that $g\mapsto [T,g,T]$ is an isomorphism $G_{n(T)}\to G_T$. Also define the subgroup
\[
K_T\defeq \{[T,g,T]\mid g\in\ker(\rho_{n(T)})\} \le \Thkern_d(G_*)\text{.}
\]

\begin{theorem}\label{thrm:inherit_ICC}
Let $((G_n)_{n\in\N},(\rho_n)_{n\in\N},(\clone_k^n)_{k\le n})$ be a fully compatible $d$-ary cloning system. If all the $G_n$ are ICC then so is $\Thomp_d(G_*)$.
\end{theorem}

\begin{proof}
We need to show that every finite index subgroup $H$ of $\Thomp_d(G_*)$ has trivial center. The intersection of $H$ with any $G_T$ has finite index in $G_T$, which since $G_T\cong G_{n(T)}$ is ICC implies $Z(H\cap G_T)=\{1\}$. In particular $Z(H)\cap G_T=\{1\}$. Now we claim that every element of $Z(H)$ lies in some $G_T$, which will imply that $Z(H)=\{1\}$. Let $[T_-,g,T_+]\in Z(H)$. Since the cloning system is fully compatible, we have a short exact sequence
\[
1\to \Thkern_d(G_*) \to \Thomp_d(G_*) \stackrel{\pi}{\to} W_d \to 1
\]
for some $F_d\le W_d\le V_d$ by Lemma~\ref{lem:map_to_Vd}. The image $\pi(H)$ of $H$ in $W_d$ has finite index in $W_d$, and $W_d$ is ICC by Lemma~\ref{lem:W_ICC} since it contains $[F_d,F_d]$, so $\pi(H)$ has trivial center. Hence $\pi([T_-,g,T_+])=1$, i.e., $[T_-,g,T_+]\in \Thkern_d(G_*)$. Since $\Thkern_d(G_*)$ is a direct union of subgroups of the $G_T$, we conclude that indeed every element of $Z(H)$ lies in some $G_T$, and we are done.
\end{proof}

If the $G_n$ are not ICC, there is still quite a lot we can say. First we need the following useful lemma:

\begin{lemma}\label{lem:embed_pure_parts}
Let $((G_n)_{n\in\N},(\rho_n)_{n\in\N},(\clone_k^n)_{k\le n})$ be a fully compatible $d$-ary cloning system. If $T'$ is an expansion of $T$ then $K_T\le K_{T'}$.
\end{lemma}

\begin{proof}
Without loss of generality $T'$ is obtained from $T$ by adding a single $d$-caret to a leaf, say the $k$th leaf. Then for any $[T,g,T]\in K_T$ we have $[T,g,T]=[T',(g)\clone_k^{n(T)},T']$, which lies in $K_{T'}$ since full compatibility ensures that $g\in\ker(\rho_n)$ implies $(g)\clone_k^n\in\ker(\rho_{n+d-1})$.
\end{proof}

\begin{definition}[Diverse]\label{def:diverse}
Call a cloning system \emph{diverse} if for all sufficiently large $n$ we have
\[
\bigcap\limits_{k=1}^n(G_n)\clone_k^n = \{1\} \text{.}
\]
\end{definition}

The idea is, in a diverse cloning system, the different cloning maps $G_n\to G_{n+d-1}$ tend to send $G_n$ into different regions of $G_{n+d-1}$.

\begin{example}\label{ex:symm_diverse}
It is easy to see that the standard $d$-ary cloning system on the symmetric groups is diverse. If $\sigma\in S_{n+d-1}$ is in the image of $\symmclone_k^n$, then $\sigma(k+i)=\sigma(k)+i$ for all $1\le i\le d-1$. Hence if $\sigma$ is in the image of every $\symmclone_k^n$ it must be the identity.
\end{example}

\begin{lemma}\label{lem:no_normals}
Let $((G_n)_{n\in\N},(\rho_n)_{n\in\N},(\clone_k^n)_{k\le n})$ be a $d$-ary cloning system. Assume that it is fully compatible and diverse. If $N\le K_T$ is normal in $\Thomp_d(G_*)$ then $N=\{1\}$.
\end{lemma}

\begin{proof}
Since the cloning system is diverse, up to replacing $T$ by an expansion, we can assume that $\bigcap\limits_{k=1}^n(G_{n(T)})\clone_k^n = \{1\}$ (here $n=n(T)$). By Lemma~\ref{lem:embed_pure_parts}, $K_T\le K_{T'}$ for any expansion $T'$ of $T$, so this step retains our assumption that $N\le K_T$. Let $[T,g,T]\in N$, so in particular every conjugate of $[T,g,T]$ in $\Thomp_d(G_*)$ lies in $K_T$. Let $T_k$ be $T$ with a $d$-caret added to the $k$th leaf of $T$. Now $[T,g,T]=[T_k,(g)\clone_k^n,T_k]$ for all $1\le k\le n$, where $n=n(T)$. Conjugating by $[T_k,T_\ell]$ for arbitrary $k$ and $\ell$ gives us that $[T_\ell,(g)\clone_k^n,T_\ell]\in K_T$ for all $k$ and $\ell$. Say $[T_\ell,(g)\clone_k^n,T_\ell]=[T,g',T]$. Then $[T_\ell,(g)\clone_k^n,T_\ell]=[T_\ell,(g')\clone_\ell^n,T_\ell]$, so $(g)\clone_k^n=(g')\clone_\ell^n$. Since such a $g'$ can be found for any $\ell$, we conclude that for all $1\le k\le n$ the element $(g)\clone_k^n$ of $G_{n+d-1}$ lies in the image of every $\clone_\ell^n$. Since the cloning system is diverse, this implies $(g)\clone_k^n=1$, so by injectivity $g=1$, hence $[T,g,T]=1$.
\end{proof}

Now we can prove:

\begin{theorem}\label{thrm:diverse_ICC}
Let $((G_n)_{n\in\N},(\rho_n)_{n\in\N},(\clone_k^n)_{k\le n})$ be a $d$-ary cloning system. Assume that it is fully compatible and diverse. Then $\Thomp_d(G_*)$ is ICC.
\end{theorem}

\begin{proof}
We will use Lemma~\ref{lem:preaux_criterion}, applied to $\Thomp_d(G_*)$ and its normal subgroup $\Thkern_d(G_*)$. We need to show that $W_d\defeq \Thomp_d(G_*)/\Thkern_d(G_*)$ is ICC and that the only finitely generated subgroup of $\Thkern_d(G_*)$ that is normal in $\Thomp_d(G_*)$ is the trivial one. The fact that $W_d$ is ICC follows from Lemma~\ref{lem:W_ICC}, since $W_d$ contains $[F_d,F_d]$. For the second claim, first note that since $\Thkern_d(G_*)$ is the directed union of the $K_T$ (by full compatibility and Lemma~\ref{lem:embed_pure_parts}), every finitely generated subgroup of $\Thkern_d(G_*)$ lies in some $K_T$. Now the claim follows from Lemma~\ref{lem:no_normals}.
\end{proof}

\begin{remark}\label{rmk:subsystem}
Given a fully compatible, diverse $d$-ary cloning system on $(G_n)_{n\in\N}$, if $H_n\le G_n$ is a family of groups such that the $d$-ary cloning system on $(G_n)_{n\in\N}$ restricts to a $d$-ary cloning system on $(H_n)_{n\in\N}$, then the latter is also fully compatible and diverse. In particular if Theorem~\ref{thrm:diverse_ICC} reveals that $\Thomp_d(G_*)$ is ICC, then it also immediately reveals that any such $\Thomp_d(H_*)$ is ICC.
\end{remark}

Note that applying Theorem~\ref{thrm:diverse_ICC} to Example~\ref{ex:symm_diverse} shows that $V_d$ is ICC, though this is also easy to see without using Theorem~\ref{thrm:diverse_ICC}.

\subsection{Non-ICC Thompson-like groups}\label{sec:non_icc_examples}

Now let us point out that not every Thompson-like group arising from a cloning system is ICC, even for nice, non-pathological examples. The easiest example is the following:

Let $G$ be any group. Let $\Pi^n(G)$ be the direct product of $n$ copies of $G$. Let $\rho_n\colon \Pi^n(G)\to S_n$ be the trivial map, and for each $1\le k\le n$ let $\clone_k^n\colon \Pi^n(G)\to \Pi^{n+d-1}(G)$ be
\[
(g_1,\dots,g_n)\clone_k^n \defeq (g_1,\dots,g_{k-1},g_k,\dots,g_k,g_{k+1},\dots,g_n) \text{.}
\]
Clearly these data form a (pure) $d$-ary cloning system. For more on these examples in the $d=2$ case, see \cite[Section~6]{witzel18} and \cite{tanushevski16}. Note that for $G\ne\{1\}$, this $d$-ary cloning system is not diverse, since any $(g,\dots,g)$ lies in the image of every cloning map.

\begin{proposition}\label{prop:direct_product}
If $G$ has non-trivial center then so does $\Thomp_d(\Pi^*(G))$. In particular, in this case $\Thomp_d(\Pi^*(G))$ is not ICC.
\end{proposition}

\begin{proof}
Let $1\ne g\in Z(G)$. We claim that the non-trivial element $[1,g,1]$ is central in $\Thomp_d(\Pi^*(G))$. First note that, thanks to the definition of the cloning maps, $[1,g,1]=[T,(g,\dots,g),T]$ for all $T$ (here the number of entries in the tuple is the number of leaves of $T$). In particular for any $[T_-,(g_1,\dots,g_n),T_+]\in \Thomp_d(\Pi^*(G))$, since $g\in Z(G)$ we have
\begin{align*}
&[T_-,(g_1,\dots,g_n),T_+][1,g,1][T_+,(g_1^{-1},\dots,g_n^{-1}),T_-] \\
&= [T_-,(g_1,\dots,g_n),T_+][T_+,(g,\dots,g),T_+][T_+,(g_1^{-1},\dots,g_n^{-1}),T_-]\\ 
&= [T_-,(g,\dots,g),T_-]=[1,g,1]\text{,}
\end{align*}
so $[1,g,1]$ is central.
\end{proof}

\section{Examples of ICC Thompson-like groups}\label{sec:examples}

Now that we have criteria that ensure $\Thomp_d(G_*)$ is ICC, let us discuss some concrete examples. Some of these $d$-ary cloning system examples have not appeared before in the literature (at least for $d>2$), so we will point out when a given example is new.

\subsection{Direct products with injective endomorphisms}\label{sec:twisted_direct_products}

Let $G$ be a (countable) group and
\[
\phi_1,\dots,\phi_d\colon G\to G
\]
a family of injective endomorphisms. Let $\Pi^n(G)$ be the direct product of $n$ copies of $G$ and let $\rho_n\colon \Pi^n(G)\to S_n$ be trivial. For $1\le k\le n$ let $\clone_k^n\colon \Pi^n(G)\to \Pi^{n+d-1}(G)$ be
\[
(g_1,\dots,g_n)\clone_k^n = (g_1,\dots,g_{k-1},\phi_1(g_k),\dots,\phi_d(g_k),g_{k+1},\dots,g_n) \text{.}
\]
It is straightforward to check that this is a (pure) $d$-ary cloning system. In the $d=2$ case these kinds of examples are essentially due to Tanushevski \cite{tanushevski16}, and were found later to fit into the cloning system framework. For $d>2$ we think this has technically not appeared before, but it is an obvious generalization. See Figure~\ref{fig:direct_product} for an example of an expansion using this cloning system. If the $\phi_i$ are all the identity, then this is the example from Subsection~\ref{sec:non_icc_examples}.

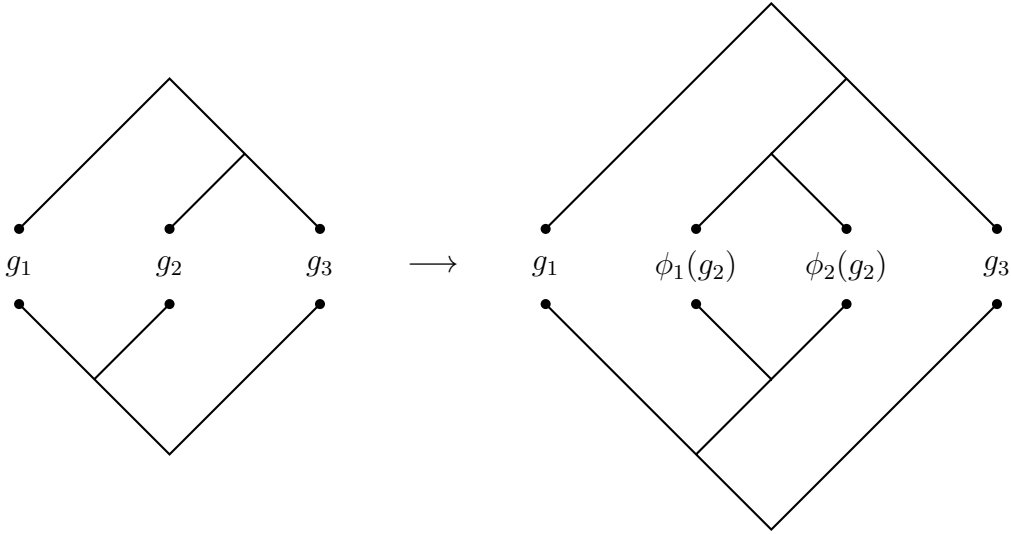
\begin{figure}[htb]
 \begin{tikzpicture}[line width=0.8pt]\centering
  \draw (0,0) -- (2,2) -- (4,0)   (3,1) -- (2,0);
	\filldraw (0,0) circle (1.5pt);
	\filldraw (2,0) circle (1.5pt);
	\filldraw (4,0) circle (1.5pt);
	\node at (0,-.5) {$g_1$};
	\node at (2,-.5) {$g_2$};
	\node at (4,-.5) {$g_3$};
	\draw (0,-1) -- (2,-3) -- (4,-1)   (1,-2) -- (2,-1);
	\filldraw (0,-1) circle (1.5pt);
	\filldraw (2,-1) circle (1.5pt);
	\filldraw (4,-1) circle (1.5pt);
	
	\node at (5.5,-.5) {$\longrightarrow$};
	
	\begin{scope}[xshift=8cm,yshift=1cm]	
   \draw (-1,-1) -- (2,2) -- (5,-1)   (3,1) -- (2,0)   (1,-1) -- (2,0) -- (3,-1);
	 \filldraw (-1,-1) circle (1.5pt);
	 \filldraw (1,-1) circle (1.5pt);
	 \filldraw (3,-1) circle (1.5pt);
	 \filldraw (5,-1) circle (1.5pt);
	 \node at (-1,-1.5) {$g_1$};
	 \node at (1,-1.5) {$\phi_1(g_2)$};
	 \node at (3,-1.5) {$\phi_2(g_2)$};
	 \node at (5,-1.5) {$g_3$};
	 \draw (-1,-2) -- (2,-5) -- (5,-2)   (1,-4) -- (2,-3)   (1,-2) -- (2,-3) -- (3,-2);
	 \filldraw (-1,-2) circle (1.5pt);
	 \filldraw (1,-2) circle (1.5pt);
	 \filldraw (3,-2) circle (1.5pt);
	 \filldraw (5,-2) circle (1.5pt);
	\end{scope}
 \end{tikzpicture}
 \caption{An example of an expansion using the ($2$-ary) cloning system defined above on $(\Pi^n(G))_{n\in\N}$. Here we draw $(T_-,(g_1,\dots,g_n),T_+)$ by drawing $T_+$ upside-down and below $T_-$, with the $g_i$ laid out between the leaves. This makes the expansion move look quite natural. Since these pictures differ by an expansion, they represent the same element of $\Thomp_2(\Pi^*(G))$.}\label{fig:direct_product}
\end{figure}

If $G$ is ICC then so is every $\Pi^n(G)$, so Theorem~\ref{thrm:inherit_ICC} says $\Thomp_d(\Pi^*(G))$ is ICC. This already provides a wealth of examples. If $G$ is not ICC, we can still use Theorem~\ref{thrm:diverse_ICC} to produce many examples, if we impose one extra assumption. Assume that the intersection of the images of all the $\phi_i$ is trivial. Now we claim that the cloning system is diverse. Indeed, for any $n\ge d$ suppose $(g_1,\dots,g_{n+d-1})$ lies in the image of every $\clone_k^n$. Then $g_d,\dots,g_n$ lie in the intersection of the images of every $\phi_i$, which means they are all trivial. Since $g_d=1$ and $\phi_d$ is injective, $g_1=\cdots=g_{d-1}=1$. Since $g_n=1$ and $\phi_1$ is injective, $g_{n+1}=\cdots=g_{n+d-1}=1$. We conclude $(g_1,\dots,g_{n+d-1})=(1,\dots,1)$, so the cloning system is diverse.

To summarize, Theorem~\ref{thrm:diverse_ICC} shows that $\Thomp_d(\Pi^*(G))$ is ICC, for any choice of $G$ and $\phi_i$, with the images of the $\phi_i$ intersecting trivially. Hence in this case $\vN(\Thomp_d(\Pi^*(G)))$ is a type $\II_1$ factor. We emphasize again that this works even if $G$ is not ICC.

\begin{example}
Let us point out some concrete examples of non-ICC groups admitting such $\phi_i$. These were suggested by Jim Belk, Yves Cornulier, Anthony Genevois, and Luc Guyot in answers to a question on mathoverflow.com \cite{mathoverflow21}. First we can take $G=\Z^\infty$ (the infinite direct sum), and it is easy to find injective endomorphisms $\phi_i\colon \Z^\infty\to\Z^\infty$ whose images intersect trivially. Hence, even though $\Z^\infty$ is abelian, and so essentially as far as possible from being ICC, nonetheless $\Thomp_d(\Pi^*(\Z^\infty))$ is ICC. For some finitely generated (and even finitely presented, indeed of type $\F_\infty$) examples we turn to Thompson-like groups themselves. The key is that for any $G$ containing $G\times G$ as a subgroup, it is possible to construct injective endomorphisms as above (for any $d$), so we just need a group $G$ that is finitely generated, has non-trivial center, and contains $G\times G$. For example if $G=F_d\times\Z$ this holds. Similarly one can use the \emph{ribbon Thompson group} $RV$ (see \cite[Subsection~3.5.3]{thumann17}), which has non-trivial center and contains $RV\times RV$. Another example is a group of the form $\Thomp_d(\Pi^*(H))$ itself, but built using the identity $H\to H$ for every $\phi_i$ (and such that $Z(H)\ne \{1\}$). Finally one can take $G$ to be the centralizer in $V$ of any finite order element. In all of these examples of $G$, we have that $G$ has non-trivial center and contains $G\times G$. Hence $G$ is not ICC but $\Thomp_d(\Pi^*(G))$ nonetheless is.
\end{example}

\subsection{(Pure) braid groups}\label{sec:braid}

The Brin--Dehornoy braided Thompson group $bV$, introduced independently by Brin \cite{brin07} and Dehornoy \cite{dehornoy06}, is a braided analog of Thompson's group $V$ that arises from a cloning system on the family of braid groups. Braid groups are not ICC, as they have non-trivial center, but as we will see in this subsection $bV$ is ICC. If one uses pure braid groups instead of braid groups, one gets the braided Thompson group $bF$, first introduced by Brady--Burillo--Cleary--Stein \cite{brady08}, and we will see that $bF$ is ICC as well. Historically speaking, the $2$-ary cloning system on the braid groups is the prototypical example of a cloning system, and $bV$ is the prototypical example of a corresponding Thompson-like group, as the impetus for the work in \cite{witzel18} was to generalize the braided case.

Let us recall the construction, and in fact we will do the $d$-ary analog, of the braided Higman--Thompson groups $bV_d$. These are a natural generalization, considered by Aroca and Cumplido in \cite{aroca21} and by Skipper and Wu in \cite{skipperwu}. These groups can be phrased in terms of $d$-ary cloning systems on the family of braid groups $(B_n)_{n\in\N}$. Numbering the strands of an element $b$ of $B_n$ from left to right at the bottom, $\clone_k^n$ sends $b$ to the element of $B_{n+d-1}$ obtained by replacing the $k$th strand with $d$ parallel strands. The maps $\rho_n\colon B_n\to S_n$ are the standard projections tracking which strand goes where. For $d=2$, it was discussed in \cite[Remark~2.10]{witzel18} that this forms a cloning system, and it is easy to generalize this to all $d$. In fact, in both \cite{aroca21} and \cite{skipperwu} a more general setup is used, which amounts to considering a $d$-ary cloning system on $(B_n\wr H)_{n\in\N}$ for any $H\le B_d$. For now we will just stick to the $bV_d=\Thomp_d(B_*)$ examples, i.e., when $H=\{1\}$ (since when $H\ne\{1\}$ the $d$-ary cloning systems are not always fully compatible). We can also do the exact same procedure to the family of pure braid groups $(PB_n)_{n\in\N}$, which yields groups denoted by $bF_d=\Thomp_d(PB_*)$. See Figure~\ref{fig:braid} for an example of expansion using the ($2$-ary) cloning system yielding $bV$.

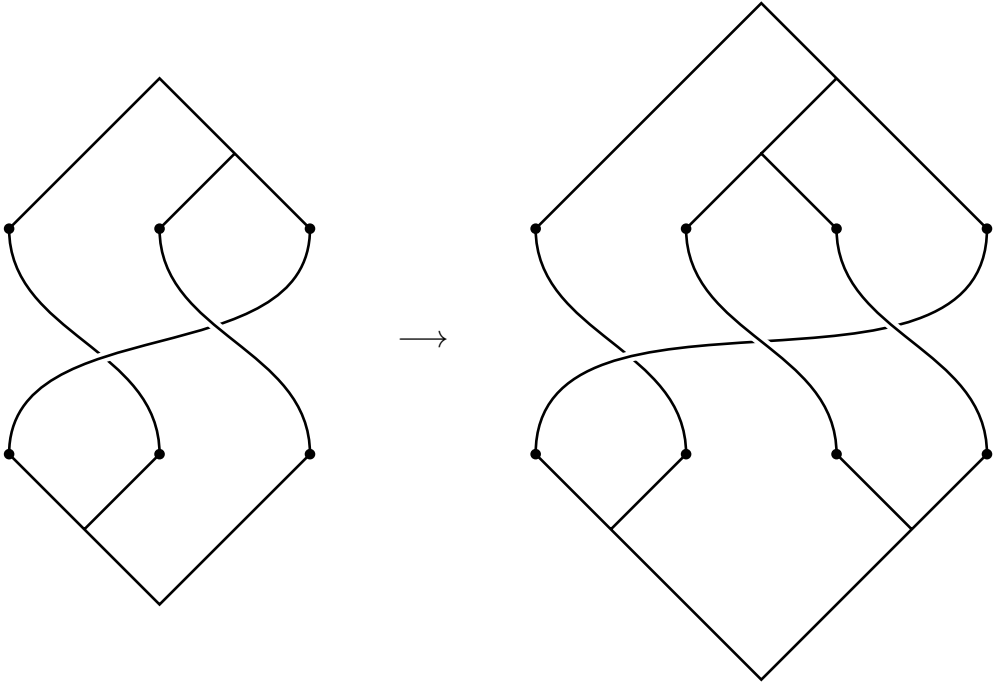
\begin{figure}[htb]
 \begin{tikzpicture}[line width=1pt]\centering
  \draw (0,0) -- (2,2) -- (4,0)   (3,1) -- (2,0);
	
	\draw[white, line width=4pt]
  (0,0) to [out=-90, in=90] (2,-3);
  \draw
  (0,0) to [out=-90, in=90] (2,-3);
	
	\draw[white, line width=4pt]
  (4,0) to [out=-90, in=90] (0,-3);
  \draw
  (4,0) to [out=-90, in=90] (0,-3);
	
  \draw[white, line width=4pt]
  (2,0) to [out=-90, in=90] (4,-3);
  \draw
  (2,0) to [out=-90, in=90] (4,-3);
	
	\filldraw (0,0) circle (1.5pt);
	\filldraw (2,0) circle (1.5pt);
	\filldraw (4,0) circle (1.5pt);
	
	\draw (0,-3) -- (2,-5) -- (4,-3)   (1,-4) -- (2,-3);
	\filldraw (0,-3) circle (1.5pt);
	\filldraw (2,-3) circle (1.5pt);
	\filldraw (4,-3) circle (1.5pt);
	
	\node at (5.5,-1.5) {$\longrightarrow$};
	
	\begin{scope}[xshift=8cm,yshift=1cm]	
   \draw (-1,-1) -- (2,2) -- (5,-1)   (3,1) -- (2,0)   (1,-1) -- (2,0) -- (3,-1);
	 
   \draw[white, line width=4pt]
   (-1,-1) to [out=-90, in=90] (1,-4);
   \draw
   (-1,-1) to [out=-90, in=90] (1,-4);
	
   \draw[white, line width=4pt]
   (5,-1) to [out=-90, in=90] (-1,-4);
   \draw
   (5,-1) to [out=-90, in=90] (-1,-4);
	
   \draw[white, line width=4pt]
   (1,-1) to [out=-90, in=90] (3,-4);
   \draw
   (1,-1) to [out=-90, in=90] (3,-4);
	
   \draw[white, line width=4pt]
   (3,-1) to [out=-90, in=90] (5,-4);
   \draw
   (3,-1) to [out=-90, in=90] (5,-4);
	
	 \filldraw (-1,-1) circle (1.5pt);
	 \filldraw (1,-1) circle (1.5pt);
	 \filldraw (3,-1) circle (1.5pt);
	 \filldraw (5,-1) circle (1.5pt);
	
	 \draw (-1,-4) -- (2,-7) -- (5,-4)   (0,-5) -- (1,-4)   (3,-4) -- (4,-5);
	 \filldraw (-1,-4) circle (1.5pt);
	 \filldraw (1,-4) circle (1.5pt);
	 \filldraw (3,-4) circle (1.5pt);
	 \filldraw (5,-4) circle (1.5pt);
	\end{scope}
 \end{tikzpicture}
 \caption{An example of an expansion using the ($2$-ary) cloning system defined above on $(B_n)_{n\in\N}$. Here we draw $(T_-,b,T_+)$ by drawing $T_+$ upside-down and below $T_-$, with the strands of braid $b$ connecting the leaves. This makes the expansion move look quite natural. Since these pictures differ by an expansion, they represent the same element of $bV=\Thomp_2(B_*)$.}\label{fig:braid}
\end{figure}

\begin{proposition}\label{prop:Vbr_ICC}
The $d$-ary cloning system on the braid groups yielding $bV_d$ is fully compatible and diverse, and so $bV_d$ is ICC. The same is true for the pure braid groups, so $bF_d$ is ICC. Hence $\vN(bV_d)$ and $\vN(bF_d)$ are type $\II_1$ factors.
\end{proposition}

\begin{proof}
These $d$-cloning systems are fully compatible by construction (the one on the pure braid groups is even pure), so we need to prove they are diverse. By Remark~\ref{rmk:subsystem}, we can just discuss $bV_d$, and $bF_d$ will follow for free. A braid in the image of $\clone_k^n$ has its $k$th through $(k+d-1)$st strands parallel to each other, so a braid in the image of every $\clone_k^n$ has all of its strands parallel. The only such braid is the identity, so indeed the cloning systems are diverse. Now Theorem~\ref{thrm:diverse_ICC} says $bV_d$ and $bF_d$ are ICC.
\end{proof}

As a remark, the fact that $bV=bV_2$ is ICC can already be gleaned from \cite{zaremsky18}. Indeed, it follows from \cite{zaremsky18} that $bV$ has no proper finite index subgroups, and $Z(bV)=\{1\}$.

\subsection{Upper triangular matrix groups}\label{sec:mtx}

Consider the family $(B_n(R))_{n\in\N}$ of groups $B_n(R)$ of invertible $n$-by-$n$ upper triangular matrices over a (countable) ring $R$. (For simplicity, let us assume $R$ is commutative, though this will not be crucial for some of what follows.) In \cite[Section~7]{witzel18} a pure ($2$-ary) cloning system on this family is described, which has an obvious $d$-ary analog. Very roughly, the idea is that the $k$th cloning map applied to a matrix $A$ replaces the $(k,k)$-entry $a_{k,k}$ with a $2$-by-$2$ matrix $\begin{pmatrix}a_{k,k}&0\\0&a_{k,k}\end{pmatrix}$, and then adds a new row and new column to accommodate this. The new column just duplicates the original $k$th column (outside this new $2$-by-$2$ matrix) and the new row is all zeros (outside this new $2$-by-$2$ matrix). The obvious $d$-ary analog is to blow up $a_{k,k}$ to a $d$-by-$d$ diagonal matrix with $a_{k,k}$ all down the diagonal, $d-1$ new columns duplicating the original, and $d-1$ new mostly-zero rows.

For more precision about the definition of the $2$-ary cloning maps, see \cite[Section~7]{witzel18}. For the sake of intuition, let us draw a picture of the cloning map $\clone_3^5$ applied to an example matrix, in the $3$-ary case:
\[
\begin{pmatrix} 1&2&3&4&5\\0&6&7&8&9\\0&0&10&11&12\\0&0&0&13&14\\0&0&0&0&15\end{pmatrix}\clone_3^5 =
\begin{pmatrix} 1&2&3&3&3&4&5\\0&6&7&7&7&8&9\\0&0&10&0&0&0&0\\0&0&0&10&0&0&0\\0&0&0&0&10&11&12\\0&0&0&0&0&13&14\\0&0&0&0&0&0&15\end{pmatrix}
\]
It is proved in \cite[Lemma~7.1]{witzel18} that these cloning maps define a pure cloning system in the $2$-ary case. The $d$-ary case has not been considered in the literature until now, but it is easy to verify that this forms a pure $d$-ary cloning system, for similar reasons as in the $d=2$ case. The following is immediate from the definition:

\begin{observation}\label{obs:cloned_mtx}
Let $A\in B_{n+d-1}(R)$ such that $A$ lies in the image of $\clone_k^n$. Then the $(i,i)$-entries $A_{i,i}$ coincide for all $k\le i\le k+d-1$, and for any $k\le i<k+d-1$ every non-diagonal entry of the $i$th row is $0$. \qed
\end{observation}

If the group of units $R^\times$ is non-trivial, say $1\ne r\in R^\times$ then it is easy to see (similarly to the example in Proposition~\ref{prop:direct_product}) that the element $[1,rI_n,1]$ of $\Thomp_d(B_*(R))$ is non-trivial and central, so $\Thomp_d(B_*(R))$ is not ICC. However, it easy to modify this example to get a family of Thompson-like groups that are ICC. Let $Z_n(R)\le B_n(R)$ be the group of homotheties, i.e., scalar multiples of the identity, so $Z_n(R)\cong (R^\times)^n$. Let $\overline{B}_n(R)\defeq B_n(R)/Z_n(R)$.

\begin{lemma}
The above pure $d$-ary cloning system on $(B_n(R))_{n\in\N}$ naturally induces a pure $d$-ary cloning system on $(\overline{B}_n(R))_{n\in\N}$ via $(AZ_n(R))\clone_k^n \defeq (A)\clone_k^n Z_{n+d-1}(R)$.
\end{lemma}

\begin{proof}
Since $\clone_k^n$ takes $Z_n(R)$ into $Z_{n+d-1}(R)$, the induced $d$-ary cloning maps are well defined. Since we are in the pure situation, to be a $d$-ary cloning system the only thing to check is the ``product of clonings'' axiom, which clearly holds.
\end{proof}

This $d$-ary cloning system on the $\overline{B}_n(R)$ is new for $d>2$, and for $d=2$ this example was essentially considered in \cite[Section~8]{witzel18}.

\begin{proposition}
The above $d$-ary cloning system on $(\overline{B}_n(R))_{n\in\N}$ is diverse, and so $\Thomp_d(\overline{B}_*(R))$ is ICC. Hence $\vN(\Thomp_d(\overline{B}_*(R)))$ is a type $\II_1$ factor.
\end{proposition}

\begin{proof}
Suppose $AZ_{n+d-1}(R) \in \overline{B}_{n+d-1}(R)$ is in the image of every $\clone_k^n$. This means that for every $1\le k\le n$, $A\in B_{n+d-1}(R)$ is a product of something in the image of $\clone_k^n$ times something in $Z_{n+d-1}(R)$. If a matrix lies in the image of $\clone_k^n$ then by Observation~\ref{obs:cloned_mtx} its $(i,i)$-entries coincide for all $k\le i\le k+d-1$, and for any $k\le i<k+d-1$, every non-diagonal entry of the $i$th row is $0$. These properties are preserved under multiplication by an element of $Z_{n+d-1}(R)$, so in particular they hold for $A$. Since these properties hold for all $1\le k\le n$, in fact all the diagonal entries of $A$ coincide and all the non-diagonal entries of $A$ are $0$. Hence $A\in Z_{n+d-1}(R)$, so the intersection of the images of all the $\clone_k^n$ is trivial in $\overline{B}_{n+d-1}(R)$. This shows the $d$-ary cloning system is diverse, and of course it is fully compatible, so $\Thomp_d(\overline{B}_*(R))$ is ICC by Theorem~\ref{thrm:diverse_ICC}.
\end{proof}

Another family of upper triangular matrix groups that was outfitted with a cloning system in \cite[Section~7]{witzel18} is the family of Abels groups, essentially due to Abels and studied in depth by Abels and Brown in \cite{abels87}. The $n$th \emph{Abels group} $\Ab_n$ (for some prime $p$) is the subgroup of $B_{n+1}(\Z[1/p])$ consisting of those upper triangular matrices whose $(1,1)$-entry and $(n+1,n+1)$-entry are both $1$. The above $d$-ary cloning system on $(B_n(\Z[1/p]))_{n\in\N}$ clearly restricts to a $d$-ary cloning system on $(\Ab_n)_{n\in\N}$, and so we get a Thompson-like group $\Thomp_d(\Ab_*)$.

The original motivation for constructing $\Thomp(\Ab_*)$ comes from finiteness properties of groups: $\Ab_n$ is of type $\F_{n-1}$ but not $\F_n$, but it turns out $\Thomp(\Ab_*)$ is of type $\F_\infty$. For our purposes here, it provides another example of a diverse cloning system and hence an ICC Thompson-like group:

\begin{proposition}
The above $d$-ary cloning system on $(\Ab_n)_{n\in\N}$ is diverse, and so $\Thomp_d(\Ab_*)$ is ICC. Hence $\vN(\Thomp_d(\Ab_*))$ is a type $\II_1$ factor.
\end{proposition}

\begin{proof}
If a matrix in $\Ab_{n+d-1}$ lies in the image of $\clone_k^n$ then by Observation~\ref{obs:cloned_mtx} its $(i,i)$-entries coincide for all $k\le i\le k+d-1$, and for any $k\le i<k+d-1$, every non-diagonal entry of the $i$th row is $0$. Since the $(1,1)$-entry of any matrix in $\Ab_{n+d-1}$ is $1$, this first observation shows that if a matrix lies in the image of every $\clone_k^n$, then each of its diagonal entries is $1$. The second observation shows that if a matrix lies in the image of every $\clone_k^n$ then every off-diagonal entry is $0$. We conclude that the identity is the only such matrix in $\Ab_{n+d-1}$. This proves diversity, and full compatibility is obvious, so Theorem~\ref{thrm:diverse_ICC} says $\Thomp_d(\Ab_*)$ is ICC.
\end{proof}

Note that $\Ab_n$ has non-trivial center (namely the subgroup of matrices differing from the identity only in the top right entry), hence is not ICC, so we could not have simply used Theorem~\ref{thrm:inherit_ICC}. As another remark, instead of requiring both the top and the bottom diagonal entry to be $1$, one could just require the top diagonal entry to be $1$ (or just the bottom). This family of groups would also yield an ICC Thompson-like group, by the same argument.

\subsection{R\"over--Nekrashevych groups}\label{sec:nekr}

The main motivation for introducing $d$-ary cloning systems in \cite{skipper21} was to extend the definition of ($2$-ary) cloning systems from \cite{witzel18} to include R\"over--Nekrashevych groups. Let us briefly recall this family of groups. Let $\tree_d$ be the rooted regular $d$-ary tree. Note that $\Aut(\tree_d)\cong S_d\wr \Aut(\tree_d)$, since an automorphism of $\tree_d$ is determined by how it permutes the $d$ vertices adjacent to the root together with what it does to each copy of $\tree_d$ whose root is one of these vertices. A subgroup $G\le \Aut(\tree_d)$ is called \emph{self-similar} if under the above isomorphism $\Aut(\tree_d)\to S_d\wr \Aut(\tree_d)$, the image of $G$ lies in $S_d\wr G$. Now for any self-similar $G$, one can put a $d$-ary cloning system on $(S_n \wr G)_{n\in\N}$, and the resulting Thompson-like group $\Thomp_d(S_*\wr G)$ is called a \emph{R\"over--Nekrashevych group}. See \cite{skipper21} for more details.

The $d$-ary cloning system above is not fully compatible. For example, $\clone_1^1\colon G\to S_d\wr G$ is exactly the restriction of the isomorphism $\Aut(\tree_d)\to S_d\wr \Aut(\tree_d)$ to $G\to S_d\wr G$, and while $\rho_1$ is trivial, $\rho_d$ is not, namely $\rho_d$ is the natural projection $S_d\wr G\to S_d$. Since the $d$-ary cloning system is not fully compatible, our sufficient conditions here for ICC do not apply. However, these groups are specific enough that it is not too difficult to tell that they are ICC:

\begin{proposition}
For any self-similar $G$, the R\"over--Nekrashevych group $\Thomp_d(S_*\wr G)$ is ICC. Hence $\vN(\Thomp_d(S_*\wr G))$ is a type $\II_1$ factor.
\end{proposition}

\begin{proof}
By \cite[Theorem~9.11]{nekrashevych04}, every non-trivial normal subgroup of $\Thomp_d(S_*\wr G)$ contains the commutator subgroup. In particular any element of $\Thomp_d(S_*\wr G)$ with finite index centralizer must commute with every element of the commutator subgroup. Viewing $\Thomp_d(S_*\wr G)$ as a group of self-homeomorphisms of the $d$-ary Cantor set $\partial\tree_d$, for any basic open subset $U$ there exists an element of the commutator subgroup whose fixed point set is precisely $U$ (indeed, the commutator subgroup of the Higman--Thompson group $V_d$ already has this property, and it is a subgroup). Since commuting elements stabilize each other's fixed point sets, this implies that any element with finite index centralizer must stabilize every basic open subset of $\partial\tree_d$. Since $\partial\tree_d$ is Hausdorff, the only such element is the identity, so $\Thomp_d(S_*\wr G)$ is ICC.
\end{proof}

\section{McDuff factors and inner amenability}\label{sec:mcduff}

If a group $G$ is ICC then its group von Neumann algebra $\vN(G)$ is a type $\II_1$ factor. From here we can ask whether additional properties hold, for example (in increasing strength) whether $G$ is inner amenable, whether $\vN(G)$ has property $\Gamma$, and whether $\vN(G)$ is McDuff.

In \cite{jolissaint97} it is shown that $F$ is inner amenable, and in \cite{jolissaint98} that $\vN(F)$ is McDuff. In \cite{picioroaga06} the same is shown for all the $F_d$. In \cite{haagerup17} it is shown that $T$ and $V$ are not inner amenable (and it seems likely that $T_d$ and $V_d$ are similarly not inner amenable). In this section we focus on pure, diverse $d$-ary cloning systems, and show that subject to one more condition called being ``uniform'', they are always inner amenable, and even stronger, they yield McDuff factors.

First let us show that, even without this upcoming uniformity condition, if the $G_n$ themselves yield McDuff factors $\vN(G_n)$, then in the pure case the same holds for $\Thomp_d(G_*)$. (Many of the $G_n$ we care about are not even ICC, much less do they yield McDuff factors, but this result is worth recording nonetheless.)

\begin{lemma}\label{lem:inherit_mcduff}
Let $((G_n)_{n\in\N},(\rho_n)_{n\in\N},(\clone_k^n)_{k\le n})$ be a pure $d$-ary cloning system. Assume that all the $G_n$ are ICC (so $\Thomp_d(G_*)$ is ICC) and that the resulting type $\II_1$ factors $\vN(G_n)$ are all McDuff. Then the type $\II_1$ factor $\vN(\Thomp_d(G_*))$ is also McDuff.
\end{lemma}

\begin{proof}
We will show that $\vN(S_\infty)\otimes\vN(\Thomp_d(G_*)) \cong \vN(\Thomp_d(G_*))$. First note that $\vN(S_\infty)\otimes \vN(\Thomp_d(G_*)) \cong \vN(S_\infty\times\Thomp_d(G_*))$. Now extend the pure $d$-ary cloning system on $(G_n)_{n\in\N}$ to a pure $d$-ary cloning system on $(S_\infty\times G_n)_{n\in\N}$ by having all the $d$-ary cloning maps act as the identity on the $S_\infty$ factor. Then $(\sigma,[T_-,g,T_+])\mapsto [T_-,(\sigma,g),T_+]$ yields a well defined isomorphism $S_\infty\times \Thomp_d(G_*) \to \Thomp_d(S_\infty\times G_*)$. Next observe that $\Thomp_d(S_\infty\times G_*) = \Thkern_d(S_\infty\times G_*)\rtimes F_d$, and $\Thkern_d(S_\infty\times G_*)$ is the direct union of the groups $S_\infty\times G_T$, i.e., the direct limit $\varinjlim (S_\infty\times G_T)$. We conclude that $\vN(S_\infty)\otimes\vN(\Thomp_d(G_*)) \cong (\varinjlim \vN(S_\infty\times G_T))\rtimes F_d$. We are assuming the $\vN(G_n)$ are McDuff, so this is isomorphic to $(\varinjlim \vN(G_T)) \rtimes F_d$, and reversing the above procedure we see that this is isomorphic to $\vN(\Thomp_d(G_*))$ as desired.
\end{proof}

As we have seen, $\Thomp_d(G_*)$ can be ICC without the $G_n$ being ICC, much less the $\vN(G_n)$ being McDuff, so it is of interest to tell when $\vN(\Thomp_d(G_*))$ can be a McDuff factor without relying on the $\vN(G_n)$ being McDuff, or even being type $\II_1$. This next definition will lead to a nice sufficient condition in Theorem~\ref{thrm:thomp_mcduff}.

\begin{definition}[Uniform]\label{def:uniform}
Call a $d$-ary cloning system \emph{uniform} if for all $1\le k\le n$ and all $\ell,\ell'$ satisfying $k\le \ell\le\ell'\le k+d-1$ we have $\clone_k^n \circ \clone_\ell^{n+d-1} = \clone_k^n \circ \clone_{\ell'}^{n+d-1}$.
\end{definition}

The idea behind uniformity is that, intuitively, if we clone an element and then clone a part of it that was involved in the first cloning, it doesn't matter which part of it we use. For example in the standard $2$-ary cloning system on the braid group, if we clone the $3$rd strand to create two parallel strands, and then follow that up by cloning one of the new strands, either the $3$rd or $4$th, it doesn't matter which one we clone. Either way we will end up effectively having turned the original $3$rd strand into three new parallel strands.

Before stating the next lemma, which to some extent is the point of uniformity, let us introduce some terminology. Note that the vertices of a given $d$-ary tree can be naturally labeled by finite words in the alphabet $\{0,\dots,d-1\}$. Given a finite word $v$ in this alphabet, say two $d$-ary trees $T$ and $U$ \emph{agree away from $v$} if there exists a $d$-ary tree with a leaf labeled by $v$ such that each of $T$ and $U$ can be obtained from this tree by adding a $d$-ary tree to this leaf.

\begin{lemma}\label{lem:uniform_lemma}
Let $((G_n)_{n\in\N},(\rho_n)_{n\in\N},(\clone_k^n)_{k\le n})$ be a uniform $d$-ary cloning system. Let $R_-$ be a $d$-ary tree with $n$ leaves, say with one leaf labeled by $v$. Let $R_+$ be another $d$-ary tree with $n$ leaves, one of which is also labeled by $v$. Let $T$ and $U$ be $d$-ary trees that agree away from $v$, with the same number of leaves. Then every element of the form $[R_-,g,R_+]$ commutes with $[T,1,U]$.
\end{lemma}

\begin{proof}
Since $[R_-,g,R_+]=[R_-,1,R_+][R_+,g,R_+]$, it suffices to first prove the result when $g=1$, and then separately prove the result when $R_-=R_+$. If $g=1$ then this result is a standard fact about $F_d$. Now assume $R_-=R_+$, call it $R$. Up to expanding $R$, we can assume without loss of generality that $R$, $T$, and $U$ all agree away from $v$. Now note that $T$ and $U$ are each expansions of $R$. Let $g'$ and $g''$ be such that $[R,g,R]=[T,g',T]=[U,g'',U]$. If $v$ is the $k$th leaf of $R$, then each of $g'$ and $g''$ are obtained from $g$ by first applying $\clone_k^n$ and then applying additional cloning maps. By uniformity then, $g'=g''$. Now
\[
[T,1,U][R,g,R][U,1,T]=[T,1,U][U,g',U][U,1,T]=[T,g',T]=[R,g,R] \text{,}
\]
i.e., $[R,g,R]$ commutes with $[T,1,U]$.
\end{proof}

Now that we have pinned down the key condition of uniformity, we can begin proving that it leads to certain factors being McDuff. First let us recall two key tools from Jolissaint's proof in \cite{jolissaint98} that $\vN(F)$ is McDuff, which were also key tools in Picioroaga's proof in \cite{picioroaga06} that every $\vN(F_d)$ is McDuff.

\begin{cit}\cite[Proposition~2.4]{jolissaint98}\label{cit:mcduff_criterion}
Let $G$ be a countable ICC group and $H\le G$ an ICC subgroup. Suppose that for any finite subset $E\subseteq G$ the intersection of the centralizers in $H$ of the elements of $E$ is non-abelian. Then the pair $\vN(H)\subseteq \vN(G)$ has the relative McDuff property.
\end{cit}

\begin{cit}\cite[Proposition~2.6]{jolissaint98}\label{cit:mcduff_criterion2}
Let $N$ be a McDuff factor with separable predual (for example if $N$ is the group von Neumann algebra of a countable group). Let $G$ be a countable amenable group (for example $\Z$) and $\alpha\colon G\to \Aut(N)$ an action of $G$ on $N$. Suppose $\alpha$ is \emph{centrally free}, meaning for any $1\ne g\in G$ there exists a central sequence $(a_n)_{n\in\N}$ in $N$ such that $\lim\limits_{n\to\infty}||\alpha(g)(a_n)-a_n||_2\ne 0$. Then the pair $N\subseteq N \rtimes_\alpha G$ has the relative McDuff property.
\end{cit}

\begin{remark}
Note that in Citation~\ref{cit:mcduff_criterion}, if such an $H$ exists satisfying this condition, then $H=G$ will also satisfy this condition. Hence this also gives us a sufficient condition on an ICC group $G$ (with no reference to a subgroup $H$) to ensure $\vN(G)$ is McDuff, namely that any intersection of finitely many element centralizers in $G$ is non-abelian. This is an interesting balancing act, since to be ICC the element centralizers must be ``small'' (infinite index) but this criterion requires them to also be ``big'' (have non-abelian intersections). In particular, note that Citation~\ref{cit:mcduff_criterion} cannot possibly apply if $G$ is finitely generated; in this case we can take $E$ to be a finite generating set, and then the condition would require the center of $G$ to be non-abelian, which is absurd. Thus Citation~\ref{cit:mcduff_criterion} is only potentially useful for non-finitely generated $G$.
\end{remark}

Considering $T$ and $V$ are not inner amenable, to get McDuff factors we should focus on ``$F$-like'' Thompson-like groups. In fact, the following generalization of pure $d$-ary cloning system will be enough of a restriction:

\begin{definition}[Slightly pure]\label{def:slightly_pure}
Call a $d$-ary cloning system on $(G_n)_{n\in\N}$ \emph{slightly pure} if for all $n\in\N$ and all $g\in G_n$ we have $\rho_n(g)(n)=n$ (here $\rho_n(g)\in S_n$, so this condition is saying that the permutation $\rho_n(g)$ of $\{1,\dots,n\}$ should fix $n$). Note that in a pure $d$-ary cloning system, $\rho_n(g)$ fixes every $i\in\{1,\dots,n\}$, so slightly pure is a generalization of pure.
\end{definition}

\begin{example}[The groups $\widehat{V}_d$]\label{ex:Vhat}
Let $\widehat{S}_n\le S_n$ be the subgroup of permutations fixing $n$, so $\widehat{S}_n\cong S_{n-1}$. The standard $d$-ary cloning system on $(S_n)_{n\in\N}$ restricts to a slightly pure $d$-ary cloning system on $(\widehat{S}_n)_{n\in\N}$, and the resulting Thompson-like group $\Thomp_d(\widehat{S}_*)$ is the subgroup of $V_d$ given by all $[T_-,\sigma,T_+]$ such that, intuitively, $\sigma$ does not permute the last leaf. Let us denote $\Thomp_d(\widehat{S}_*)$ by $\widehat{V}_d$; when $d=2$ this coincides with the group denoted by $\widehat{V}$ in \cite{brin07}, for reasons similar to the explanation given in \cite{brady08} in the braided case. For $d>2$ we believe the group $\widehat{V}_d$ has technically not appeared before in the literature, but it is an obvious generalization of $\widehat{V}=\widehat{V}_2$.
\end{example}

Since for an element $[T_-,\sigma,T_+]$ of $\widehat{V}_d$, the permutation $\sigma$ does not permute the last leaf, we get a well defined homomorphism
\[
\theta\colon \widehat{V}_d\to\Z
\]
given by sending $[T_-,\sigma,T_+]$ to $\delta_r(T_-)-\delta_r(T_+)$, where for a tree $T$ we write $\delta_r(T)$ for the distance from the root to the rightmost leaf.

For the rest of the section, we assume we have a fully compatible, slightly pure $d$-ary cloning system on a family of groups $(G_n)_{n\in\N}$. Now let us introduce a subgroup of $\Thomp_d(G_*)$ that will allow us to leverage Citations~\ref{cit:mcduff_criterion} and~\ref{cit:mcduff_criterion2}. Since our $d$-ary cloning system is fully compatible we have a map $\pi\colon \Thomp_d(G_*)\to V_d$ with kernel $\Thkern_d(G_*)$. Let $D_d\le \widehat{V}_d$ be the kernel of the map $\theta$ defined above, and let
\[
D_d(G_*)\defeq \pi^{-1}(D_d)\text{,}
\]
so $D_d(G_*)$ consists of all $[T_-,g,T_+]$ such that $\delta_r(T_-)=\delta_r(T_+)$. Note that $[F_d,F_d]\le \pi(D_d(G_*))\le D_d$.

\begin{lemma}\label{lem:D_ICC}
Let $((G_n)_{n\in\N},(\rho_n)_{n\in\N},(\clone_k^n)_{k\le n})$ be a fully compatible, slightly pure $d$-ary cloning system. Assume either that all the $G_n$ are ICC, or that the $d$-ary cloning system is diverse and uniform. Then the group $D_d(G_*)$ defined above is ICC.
\end{lemma}

\begin{proof}
First assume all the $G_n$ are ICC. The argument proceeds very similarly to the proof of Theorem~\ref{thrm:inherit_ICC}. As in that proof, it suffices to show that for any finite index subgroup $H$ of $D_d(G_*)$, every element of the center $Z(H)$ lies in some $G_T$. Under the map $\pi\colon D_d(G_*)\to D_d$, the image of $H$ contains $[F_d,F_d]$, and hence is ICC by Lemma~\ref{lem:W_ICC}. Thus every element of $Z(H)$ lies in $\Thkern_d(G_*)$, and hence in some $G_T$ as desired.

Now assume the $d$-ary cloning system is diverse and uniform (and the $G_n$ are not necessarily ICC). Since $\pi(D_d(G_*))$ contains $[F_d,F_d]$, it is ICC by Lemma~\ref{lem:W_ICC}. Hence by Lemma~\ref{lem:preaux_criterion}, it suffices to show that the only finitely generated subgroup of $\Thkern_d(G_*)$ that is normal in $D_d(G_*)$ is the trivial one. The argument is similar to the one in Lemma~\ref{lem:no_normals} showing that this holds for $\Thomp_d(G_*)$. As in that proof, we reduce to the situation where we have an element $[T,g,T]$, and every conjugate of $[T,g,T]$ in $D_d(G_*)$ (but now not necessarily in $\Thomp_d(G_*)$) lies in $K_T$. We want to show that $g=1$. First expand $[T,g,T]$ so that without loss of generality $g\in G_n$ lies in the image of $\clone_{n-(d-1)}^{n-(d-1)}$. Now the argument from the proof of Lemma~\ref{lem:no_normals} shows that for all $1\le k<n$ the element $(g)\clone_k^n$ of $G_{n+d-1}$ lies in the image of $\clone_\ell^n$ for all $1\le \ell<n$. Note that these bounds do not include $k=n$ or $\ell=n$, since we must work in $D_d(G_*)$. However, because $g$ lies in the image of $\clone_{n-(d-1)}^{n-(d-1)}$ and the $d$-ary cloning system is uniform, $(g)\clone_{n-1}^n=(g)\clone_n^n$, and so in fact this also lies in the image of $\clone_n^n$. Having shown that $(g)\clone_{n-1}^n$ lies in the image of $\clone_\ell^n$ for all $1\le \ell\le n$, we conclude that $(g)\clone_{n-1}^n=1$ since the $d$-ary cloning system is diverse, and so by injectivity $g=1$ as desired.
\end{proof}

Note that $\Thomp_d(G_*)\cong D_d(G_*)\rtimes \Z$, where the map $\Thomp_d(G_*)\to\Z$ is the composition of $\pi\colon \Thomp_d(G_*)\to \widehat{V}_d$ with $\theta\colon\widehat{V}_d\to \Z$. Thus we can view $\Thomp_d(G_*)$ as an internal semidirect product $\Thomp_d(G_*)= D_d(G_*)\rtimes \Z$ by choosing for the generator of $\Z$ any element of $\Thomp_d(G_*)$ with $\theta$ value $1$, for instance the standard generator $x_0$. The decomposition
\[
\Thomp_d(G_*)=D_d(G_*)\rtimes\langle x_0\rangle
\]
will be important in the coming proof.

\begin{theorem}\label{thrm:thomp_mcduff}
Let $((G_n)_{n\in\N},(\rho_n)_{n\in\N},(\clone_k^n)_{k\le n})$ be a fully compatible, slightly pure, uniform $d$-ary cloning system. Assume either that all the $G_n$ are ICC, or that the $d$-ary cloning system is diverse (so in either case $\Thomp_d(G_*)$ is ICC). Then the pair $\vN(D_d(G_*))\subseteq \vN(\Thomp_d(G_*))$ has the relative McDuff property, so in particular $\vN(\Thomp_d(G_*))$ is a McDuff factor and $\Thomp_d(G_*)$ is inner amenable.
\end{theorem}

\begin{proof}
First we explain why $\vN(D_d(G_*))$ is McDuff. We will use Citation~\ref{cit:mcduff_criterion}, which applies since $D_d(G_*)$ is ICC by Lemma~\ref{lem:D_ICC}. Let $E$ be a finite subset of $D_d(G_*)$. Let $m\in\N$ be such that every element of $E$ can be represented by some $[T_-,g,T_+]$ with $\delta_r(T_-)=\delta_r(T_+)=m$. Note that the rightmost leaves of any such $T_-$ and $T_+$ are each labeled by $(d-1)^m$. Let $T$ and $U$ be any trees that agree away from $(d-1)^m$, with the same number of leaves. By Lemma~\ref{lem:uniform_lemma}, $[T,1,U]$ commutes with every element of $E$. Since the collection of all such $[T,1,U]$ certainly includes non-commuting elements, this shows that the intersection of the centralizers of elements of $E$ is non-abelian, so Citation~\ref{cit:mcduff_criterion} says $\vN(D_d(G_*))$ is McDuff.

Now we will apply Citation~\ref{cit:mcduff_criterion2} to the decomposition $\Thomp_d(G_*)=D_d(G_*)\rtimes\langle x_0\rangle$, to prove that the pair $\vN(D_d(G_*))\subseteq \vN(\Thomp_d(G_*))$ has the relative McDuff property. The conjugation action of $\langle x_0\rangle$ on $D_d(G_*)$ induces an action $\alpha$ of $\langle x_0\rangle$ on the McDuff factor $\vN(D_d(G_*))$. Let $x_0^k$ for $k\ne 0$ be an arbitrary non-trivial element of $\langle x_0\rangle$. For each $n\in\N$ let $a_n$ be a non-trivial element of $[F_d,F_d]\le D_d$ whose support lies in $(1-\frac{1}{n},1)$ (here the \emph{support} of a self-homeomorphism of $[0,1]$ is the set of points it does not fix). Then for all $n\in\N$ we have $x_0^k a_n x_0^{-k} \ne a_n$, so $\lim\limits_{n\to\infty}||\alpha(x_0^k)(a_n)-a_n||_2=\sqrt{2}\ne 0$. We claim that $(a_n)_{n\in\N}$ is a (non-trivial) central sequence in $\vN(D_d(G_*))$, which will imply that $\alpha$ is centrally free. As in the proof of \cite[Corollary~2.7]{jolissaint98}, it suffices to show that every finite subset $E$ of $D_d(G_*)$ commutes with $a_n$ for $n$ large enough, and we can use a similar argument as in the previous paragraph. Let $m\in\N$ be such that every element of $E$ can be represented by some $[T_-,g,T_+]$ with $\delta_r(T_-)=\delta_r(T_+)=m$. Note that the rightmost leaves of any such $T_-$ and $T_+$ are each labeled by $(d-1)^m$. Let $T$ and $U$ be any trees that agree away from $(d-1)^m$, with the same number of leaves. By Lemma~\ref{lem:uniform_lemma}, $[T,1,U]$ commutes with every element of $E$. For any sufficiently large $n$, $a_n$ is of this form, so $\alpha$ is centrally free and we are done.
\end{proof}

\subsection{Examples}\label{sec:mcduff_examples}

Let us point out which examples from Section~\ref{sec:examples} satisfy the conditions in Theorem~\ref{thrm:thomp_mcduff}, and hence yield McDuff factors and inner amenable groups.

\begin{example}[$\widehat{V}_d$]\label{ex:Vhat_mcduff}
As a first, somewhat surprising example, consider the $d$-ary cloning system on $(\widehat{S}_n)_{n\in\N}$ from Example~\ref{ex:Vhat}. This is fully compatible and slightly pure, and is easily seen to be diverse and uniform. Hence by Theorem~\ref{thrm:thomp_mcduff}, $\vN(\widehat{V}_d)$ is a McDuff factor, and $\widehat{V}_d$ is inner amenable. The reason we find this somewhat surprising is that $\widehat{V}_d$ is quite similar to $V_d$, for example $\widehat{V}_d$ and $V_d$ embed into each other in natural ways (the embedding $\widehat{V}_d\to V_d$ is obvious, and an embedding $V_d\to \widehat{V}_d$ can be given by, roughly, sticking each $d$-ary tree on the first leaf of a $d$-caret, and embedding $S_n$ into $\widehat{S}_{n+d-1}$ in the obvious way). But $V=V_2$ is not inner amenable, and it is reasonable to expect none of the $V_d$ are.
\end{example}

\begin{example}[Direct products with injective endomorphisms]\label{ex:twisted_products_mcduff}
Next consider $\Pi^n(G)$, with injective endomorphisms $\phi_1,\dots,\phi_d\colon G\to G$, as in Subsection~\ref{sec:twisted_direct_products}. One can check that, in order for this $d$-ary cloning system to be uniform, we would need all the $\phi_i$ to be the identity. In this case the $d$-ary cloning system is certainly not diverse, so Theorem~\ref{thrm:thomp_mcduff} will only apply directly if $G$ itself is ICC. In the case when $G$ is ICC and the $\phi_i$ are all the identity everything works: the $d$-ary cloning system is pure and uniform, and the $\Pi^n(G)$ are ICC, so Theorem~\ref{thrm:thomp_mcduff} says $\vN(\Thomp_d(\Pi^*(G)))$ is a McDuff factor, and $\Thomp_d(\Pi^*(G))$ is inner amenable. Another related example, which does not require $G$ to be ICC, is as follows. Let $\Psi^n(G)\defeq \{1\}\times\Pi^{n-1}(G) \le \Pi^n(G)$ (still in the situation where all the $\phi_i$ are the identity). The $d$-ary cloning system on $\Pi^n(G)$ restricts to one on $\Psi^n(G)$, but restricted to $\Psi^n(G)$ now it is diverse. It is also (still) pure and uniform, so Theorem~\ref{thrm:thomp_mcduff} says $\vN(\Thomp_d(\Psi^*(G)))$ is a McDuff factor, and $\Thomp_d(\Psi^*(G))$ is inner amenable. We emphasize that here $G$ can be any (countable) group. These $d$-ary cloning systems on the $\Psi^n(G)$ are new, and seem especially intriguing.
\end{example}

\begin{example}[Braided Thompson groups]
Now consider the standard $d$-ary cloning system on the pure braid groups $PB_n$ from Subsection~\ref{sec:braid}. This is easily seen to be pure, diverse, and uniform, so Theorem~\ref{thrm:thomp_mcduff} says $\vN(\Thomp_d(PB_*))$ is a McDuff factor, and $bF_d=\Thomp_d(PB_*)$ is inner amenable. For example the standard ``braided $F$'' group $bF$ is inner amenable. We consider this to be surprising, since $bF_d$ contains so many non-abelian free subgroups. Indeed, any pair of non-commuting elements of $\Thkern_d(PB_*)\le bF_d$ lie in some $PB_n$ and hence generate a copy of $F_2$ \cite{leininger10}. (As for $bV_d=\Thomp_d(B_*)$, since this is ``$V$-like'' we expect it is not inner amenable.)
\end{example}

\begin{example}[Upper triangular matrix groups]
Finally, consider the $d$-ary cloning systems on families of upper triangular matrix groups from Subsection~\ref{sec:mtx}. These are pure, and easily seen to be uniform. As we saw, for the families $(\overline{B}_n(R))_{n\in\N}$ and $(\Ab_n)_{n\in\N}$ the $d$-ary cloning system is diverse. Hence Theorem~\ref{thrm:thomp_mcduff} says $\vN(\Thomp_d(\overline{B}_*(R)))$ and $\vN(\Thomp_d(\Ab_*))$ are McDuff factors, and $\Thomp_d(\overline{B}_*(R))$ and $\Thomp_d(\Ab_*)$ are inner amenable.
\end{example}

The last example from Section~\ref{sec:examples} was the R\"over--Nekrashevych groups, which are ICC, but these $d$-ary cloning systems are not fully compatible, and tend to not be uniform either, and to begin with should be considered ``$V$-like'', so all in all we would not expect R\"over--Nekrashevych groups to be inner amenable. We expect that the proof of non-inner amenability of $V$ from \cite{haagerup17} should apply to R\"over--Nekrashevych groups when $d=2$, and in general we expect the result to be true for all $d$.

\section{Questions}\label{sec:questions}

Let us conclude with some questions that naturally arise.

\begin{question}[Non-fully compatible?]
Can we drop the assumption that the $d$-ary cloning system is fully compatible in Theorems~\ref{thrm:inherit_ICC}, \ref{thrm:diverse_ICC}, and~\ref{thrm:thomp_mcduff}? Full compatibility is crucial for our proofs, since it is required for the existence of the map $\pi\colon \Thomp_d(G_*) \to V_d$, but it would be interesting to find proofs that do not rely on this map. We suspect that full compatibility is just a useful tool, and should not be necessary for the results to still hold.
\end{question}

\begin{question}[$V$-like implies non-inner amenable?]\label{quest:non_inner_amen}
In \cite{haagerup17}, Haagerup and Olesen proved that $T$ and $V$ are not inner amenable, and so $\vN(T)$ and $\vN(V)$ are not McDuff. It is natural to ask whether $\Thomp_d(G_*)$ similarly fails to be inner amenable in certain non-pure cases, e.g., when the image of $\rho_n$ in $S_n$ does not fix any element of $\{1,\dots,n\}$. The difficulty is that Haagerup--Olesen's proof relied on a certain subgroup $\Lambda$ of $T$, which does not naturally pull back to $\Thomp_d(G_*)$ if the $\rho_n\colon G_n\to S_n$ do not split, so it is hard to tell what to expect. In the cases where the maps $\rho_n\colon G_n\to S_n$ do split, e.g., for the $d$-ary cloning systems yielding R\"over--Nekrashevych groups $\Thomp_d(S_*\wr G)$, we expect the proof of non-inner amenability should work the same way.
\end{question}

\begin{question}[Finitely generated Vaes group?]
Can $d$-ary cloning systems be used to construct an example of a finitely generated (or finitely presented, or type $\F_\infty$) ICC group that is inner amenable but whose group von Neumann algebra does not have property $\Gamma$ (and so is not McDuff)? In \cite{vaes12} Vaes constructed a non-finitely generated example, and conjectured that a finitely generated example should exist. One of the key features of $d$-ary cloning systems is their ability to produce groups with nice finiteness properties, so this seems like a promising direction. On the other hand, all our examples presented here either yield McDuff factors or seem likely to not be inner amenable (see also Question~\ref{quest:non_inner_amen}), so it is not entirely clear what to expect.
\end{question}

\begin{question}[Is $F$ a McDuff group?]
In addition to asking whether the group von Neumann algebra of a group $G$ is a McDuff factor, one can also ask whether $G$ is a \emph{McDuff group}, in the sense of Deprez--Vaes \cite{deprez18}. A group is McDuff if it admits a free ergodic probability measure preserving action on some measure space $(X,\mu)$ such that the crossed product $L^\infty(X)\rtimes G$ is a McDuff factor. It turns out that there exist McDuff groups $G$ such that $\vN(G)$ is not a McDuff factor, and even such that $\vN(G)$ does not have property $\Gamma$ \cite{kida15}. As for the converse, to the best of our knowledge it is open whether $\vN(G)$ being McDuff implies that $G$ is a McDuff group. Thus, we pose the question, are all the groups in this paper whose group von Neumann algebras are McDuff factors themselves McDuff groups? For example, is $F$ a McDuff group?
\end{question}

\begin{question}[Leverage ``non-amenable but inner amenable''?]
We now have many new examples of non-amenable groups that are inner amenable, for instance $bF$. In \cite{tucker-drob20}, Tucker-Drob provides a wealth of results about non-amenable, inner amenable groups. Can any of these results give us insight into $bF$, or into $F$? Just to single out one example of a potentially interesting result, since $bF$ is torsion-free, \cite[Theorem~9]{tucker-drob20} implies that given any finite collection of non-amenable subgroups $H_1,\dots,H_k$ of $bF$, there exists an element $g\in bF$ such that for all $1\le i\le k$ the intersection of $H_i$ with the centralizer of $g$ is non-amenable. Does this provide any interesting restrictions on the subgroup structure of $bF$?
\end{question}

\bibliographystyle{alpha}
\newcommand{\etalchar}[1]{$^{#1}$}

\end{document}